\documentclass[reqno,12pt]{amsart}
\usepackage[T1]{fontenc}
\usepackage{enumerate,delarray}
\theoremstyle{plain}    
\newtheorem{thm}{Theorem}
\newtheorem{rmk}{Remark}

\newtheorem{fact}{Fact}
\newcommand{\ra}{\rightarrow}

\begin{document}
\title{On systems of rational difference equations and periodic tetrachotomies}
\author[Frank J. Palladino]{Frank J. Palladino}
\address{Department of Mathematics, University of Rhode Island,Kingston, RI 02881-0816, USA;}
\email{frank@math.uri.edu}
\date{May 1, 2010}

\begin{abstract}
\noindent We study the following system of two rational
difference equations
$$x_{n}=\frac{\beta_{k}x_{n-k} +\gamma_{k}y_{n-k}}
{1+\sum_{j=1}^{\ell}B_{j}x_{n-j} + \sum_{j=1}^{\ell}C_{j}y_{n-j}},\quad n\in\mathbb{N},$$
$$y_{n}=\frac{\delta_{k}x_{n-k} +\epsilon_{k}y_{n-k}}
{1+\sum_{j=1}^{\ell}D_{j}x_{n-j} + \sum_{j=1}^{\ell}E_{j}y_{n-j}},\quad n\in\mathbb{N},$$
with nonnegative parameters and nonnegative initial conditions. We assume that $B_{j}=C_{j}=D_{j}=E_{j}=0$ for $j=k,\; 2k,\; 3k,\; \dots$ and establish the existence of periodic
tetrachotomy behavior which depends on a $2\times 2$ matrix with entries $\beta_{k}, \gamma_{k}, \delta_{k},$ and $\epsilon_{k}$.
\end{abstract}
\maketitle

\section{Introduction}
Recently, several papers discussing rational systems in the plane have appeared in the literature. We refer particularly to \cite{systri}, \cite{syspln}, and \cite{ladassystems1}. In \cite{systri}, the authors mention a conjecture regarding periodic trichotomy behavior for some rational systems in the plane.
Given this interest in developing bifurcation results in the setting of systems of two rational difference equations, we ask the following question. ``What is the natural generalization of the periodic trichotomy behavior when we move to the setting of systems of two rational difference equations? ''\newline
It turns out that for a certain family of periodic trichotomy results the natural generalization is a periodic tetrachotomy. We use the word ``tetrachotomy'' to indicate a four way split of qualitative behaviors. This four way split arises naturally due to the added dimension. In two dimensions, the non-hyperbolic case is split into two subcases.
The reason the non-hyperbolic case is split into two subcases is a straightforward consequence of the Perron-Frobenius theorem. A nonnegative $1\times 1$ matrix with spectral radius $1$ can only have $1$ as an eigenvalue. However a nonnegative $2\times 2$ matrix with spectral radius $1$ has a choice between $1$ and $-1$ for eigenvalues. How the cases are split and what role the eigenvalues play will be made clear later in the article.\newline 

\section{A Family of Periodic Trichotomies}
To understand the essence of how rational difference equations behave, it is vital to understand the interaction between delays in the numerator and delays in the denominator. Qualitatively, one can say that when the greatest common divisor of the delays in the numerator does not divide any of the delays in the denominator, then the numerator and denominator have little interaction.
To be more specific if this occurs, then there is a nontrivial subspace of initial conditions
where the solution behaves linearly. In \cite{f2} and \cite{f3}, the author shows that the rational difference equation inheirits trichotomy behavior from the associated linear difference equation in this case.\newline
To give a demonstration of this idea consider the most basic case, namely the rational difference equation where there is a single delay present in the numerator and every multiple of that delay is not present in the denominator.
In other words consider the rational difference equation
$$x_{n}=\frac{\beta_{k}x_{n-k}}{1+\sum^{\ell}_{j=1}B_{j}x_{n-j}},\quad n\in\mathbb{N},$$
where $B_{j}=0$ for $j=k,\; 2k,\; 3k,\; \dots$. In this case simply choose initial conditions so that if $n\not\equiv 0 \mod{k}$ then $x_{n}=0$. When initial conditions are chosen this way then induction guarantees that if $n\not\equiv 0 \mod{k}$ then $x_{n}=0$ for all $n\in\mathbb{N}$. So under this choice of initial conditions if $n\equiv 0 \mod{k}$ then 
$$x_{n}=\beta_{k}x_{n-k}.$$ From this it is already clear that when $\beta_{k}>1$ there exist unbounded solutions under an appropriate choice of initial conditions. When $\beta_{k}<1$ then the map is a contraction and clearly every solution converges to zero. When $\beta_{k}=1$ the subsequences $x_{mk+a}$ must be monotone. Since bounded monotone sequences converge every solution converges to a periodic solution of not necessarily prime period $k$. Also choosing initial conditions so that $x_{n}=1$ if $n\equiv 0 \mod{k}$ and $x_{n}=0$ if $n\not\equiv 0 \mod{k}$ gives a periodic solution of prime period $k$.\newline   
Our goal in this paper will be to create an analogue of this basic trichotomy case for systems of two rational difference equations. The added dimension makes the process significantly more difficult in the boundary case mainly because we no longer have the monotonicity, which we used in the one dimensional case. To get around this difficulty we must assume that the matrix, which describes the behavior on the invariant subspace where our equation acts linearly, is Hermitian. Under this assumption monotonicity is replaced by monotonicity in norm, at which point theorems 1 and 2 from \cite{f2} are applied. Using this approach the proof goes through in many cases. The remaining Hermitian cases are handled by another type of monotonicity argument. 
Thus we obtain a tetrachotomy result in the Hermitian cases. Extending such a result to the full range of parameters is more difficult since there are several non-symmetric cases where the monotonicity breaks down. In one of these cases we cannot use the standard inner product norm, as we do in the Hermitian cases, but we give another function which depends on our parameters. The solution is monotone with respect to this function and this allows the result to be shown. The last case uses monotonicity coupled with an argument involving the limit superior and limit inferior of subsequences of our solution. 

\section{A Representation Using Vector Spaces}

Consider the system of two rational
difference equations
$$x_{n}=\frac{\beta_{k}x_{n-k} +\gamma_{k}y_{n-k}}
{1+\sum_{j=1}^{\ell}B_{j}x_{n-j} + \sum_{j=1}^{\ell}C_{j}y_{n-j}},\quad n\in\mathbb{N},$$
$$y_{n}=\frac{\delta_{k}x_{n-k} +\epsilon_{k}y_{n-k}}
{1+\sum_{j=1}^{\ell}D_{j}x_{n-j} + \sum_{j=1}^{\ell}E_{j}y_{n-j}},\quad n\in\mathbb{N},$$
with nonnegative parameters and nonnegative initial conditions. Assume that $B_{j}=C_{j}=D_{j}=E_{j}=0$ for $j=k,\; 2k,\; 3k,\; \dots$.
We find that it is useful to rewrite our system using matrix notation. We let
$$v_{n}=\left(\begin{array}{cc}
x_{n}\\               
y_{n}\\ 
\end{array}
\right),\quad A=\left(\begin{array}{cc}
\beta_{k} & \gamma_{k} \\
\delta_{k} & \epsilon_{k} \\
\end{array}\right),$$  and $$B_{n}=\left(\begin{array}{cc}
\frac{1}{1+\sum_{j=1}^{\ell}a_{j}\cdot v_{n-j}} & 0 \\
0 & \frac{1}{1+\sum_{j=1}^{\ell}q_{j}\cdot v_{n-j}} \\
\end{array}\right),$$
where 
$$a_{j}=\left(\begin{array}{cc}
B_{j}\\               
C_{j}\\ 
\end{array}
\right)\quad and\quad q_{j}=\left(\begin{array}{cc}
D_{j}\\               
E_{j}\\ 
\end{array}
\right).$$
Our system then becomes
$$v_{n}= B_{n}Av_{n-k},\quad n\in\mathbb{N}.$$
In the next few sections we prove results for systems written in this form.
\section{The contraction case}
In the first theorem of this section we prove that when the spectral radius of $A$ is less than one then every solution converges to the zero equilibirum. 
This is the contraction case of our tetrachotomy.
\begin{thm}
Consider the recursive system on $[0,\infty)^{m}$
$$v_{n}=B_{n}Av_{n-k},\quad n\in\mathbb{N},$$
where $A=(a_{ij})$ is a real $m\times m$ matrix with nonnegative entries $a_{ij}\geq 0$ and with spectral radius less than $1$. Assume that initial conditions are in $[0,\infty)^{m}$. 
Further assume that $B_{n}$ is a real $m\times m$ diagonal matrix which may depend on $n$ and on prior terms of our solution $\{v_{n}\}$, with all entries $b_{n,ii}\in [0,1]$ for all $n\in\mathbb{N}$.
Then every solution converges to the $0$ vector.
\end{thm}
\begin{proof}
Consider the system 
$$u_{n}= Au_{n-k},\quad n\in\mathbb{N}.$$
Suppose $v_n=u_n$ for $n<1$. In other words suppose that the two systems have the same initial conditions. Then the $ith$ entry of the vector $v_{n}$ is less than or equal to the $ith$ entry of the vector $u_{n}$ for all $n\in \mathbb{N}$ and for all $i\in\{1,\dots ,m\}$, 
in other words $v_{n,i}\leq u_{n,i}$. We prove this by strong induction on $n$. The initial conditions provide the base case. Suppose the result holds for $n<N$. 
$$v_{N,i}= b_{N,ii}\sum^{m}_{j=1}a_{ij}v_{N-k,j}\leq \sum^{m}_{j=1}a_{ij}v_{N-k,j}\leq \sum^{m}_{j=1}a_{ij}u_{N-k,j} = u_{N,i},$$
since $b_{N,ii}\in [0,1]$ and $a_{ij}\geq 0$ for all $i,j\in\{1,\dots ,m\}$.
Thus we have shown $v_{n,i}\leq u_{n,i}$ for all $n\in\mathbb{N}$. \newline 
It is clear that $u_{kn+b}=A^{n}u_{b}$. Now if the spectral radius of $A$ is less than one it is a well known result that $\lim_{n\ra\infty}A^{n}=0$. Of course by $0$ here we mean the zero matrix. Thus, in this case,
$\lim_{n\ra\infty} u_{n} = \left(\begin{array}{ccc}
0\\
\vdots \\               
0\\ 
\end{array}
\right)$. Since $v_{n}\in [0,\infty)^{m}$ for all $n\in\mathbb{N}$, we have $\lim_{n\ra\infty} v_{n} = \left(\begin{array}{ccc}
0\\  
\vdots \\             
0\\  
\end{array}
\right)$.
\end{proof}

The next theorem is not used to establish the tetrachotomy result however it is another general boundedness and convergence result for systems of two rational difference equations. In some sense this result also relies on having small numerators, and so belongs in this section.
\begin{thm}
Consider the $k^{th}$ order system of two rational
difference equations
$$x_n=\frac{\alpha+\sum^{k}_{i=1}\beta_{i}x_{n-i} + \sum^{k}_{i=1}\gamma_{i}y_{n-i}}{1+\sum^{k}_{j=1}B_{j}x_{n-j} + \sum^{k}_{j=1}C_{j}y_{n-j}},\quad n\in\mathbb{N},$$
$$y_n=\frac{p+\sum^{k}_{i=1}\delta_{i}x_{n-i} + \sum^{k}_{i=1}\epsilon_{i}y_{n-i}}{1+\sum^{k}_{j=1}D_{j}x_{n-j} + \sum^{k}_{j=1}E_{j}y_{n-j}},\quad n\in\mathbb{N}.$$
In particular, we assume non-negative parameters and non-negative initial conditions. We also assume that $\sum^{k}_{j=1}D_{j}=\sum^{k}_{j=1}C_{j}$. 
Note that if both sums are zero then this is clearly true, if both sums are positive and this is not the case, then we may make a change of variables so that it is true. 
However this change of variables will alter the other parameters. We further assume that $\sum^{k}_{i=1}\beta_{i}+\gamma_{i}<1$, $\sum^{k}_{i=1}\delta_{i}+\epsilon_{i}<1$, $\sum^{k}_{i=1}\beta_{i}+\delta_{i}<1$, and $\sum^{k}_{i=1}\gamma_{i}+\epsilon_{i}<1$. Then every solution converges to a finite limit. 
\end{thm}
\begin{proof}
First we prove that every solution of the system is bounded. Notice that
$$x_n=\frac{\alpha+\sum^{k}_{i=1}\beta_{i}x_{n-i} + \sum^{k}_{i=1}\gamma_{i}y_{n-i}}{1+\sum^{k}_{j=1}B_{j}x_{n-j} + \sum^{k}_{j=1}C_{j}y_{n-j}}\leq \alpha+\sum^{k}_{i=1}\beta_{i}x_{n-i} + \sum^{k}_{i=1}\gamma_{i}y_{n-i}$$
$$\leq \alpha+\left(\sum^{k}_{i=1}\beta_{i}\right)\max_{i=1,\dots , k}{x_{n-i}} + \left(\sum^{k}_{i=1}\gamma_{i}\right)\max_{i=1,\dots , k}{y_{n-i}}$$
$$\leq \alpha + \left(\sum^{k}_{i=1}\beta_{i}+\gamma_{i}\right)\max{\left(\max_{i=1,\dots , k}{x_{n-i}},\max_{i=1,\dots , k}{y_{n-i}}\right)}. $$
Also we have
 $$y_n=\frac{p+\sum^{k}_{i=1}\delta_{i}x_{n-i} + \sum^{k}_{i=1}\epsilon_{i}y_{n-i}}{1+\sum^{k}_{j=1}D_{j}x_{n-j} + \sum^{k}_{j=1}E_{j}y_{n-j}}\leq p+\sum^{k}_{i=1}\delta_{i}x_{n-i} + \sum^{k}_{i=1}\epsilon_{i}y_{n-i}$$
$$\leq p+\left(\sum^{k}_{i=1}\delta_{i}\right)\max_{i=1,\dots , k}{x_{n-i}} + \left(\sum^{k}_{i=1}\epsilon_{i}\right)\max_{i=1,\dots , k}{y_{n-i}}$$
$$\leq p + \left(\sum^{k}_{i=1}\delta_{i}+\epsilon_{i}\right)\max{\left(\max_{i=1,\dots , k}{x_{n-i}},\max_{i=1,\dots , k}{y_{n-i}}\right)}.$$
Thus we get
$$\max{(x_{n},y_{n})}\leq \max{(\alpha,p)} + \max{\left(\left(\sum^{k}_{i=1}\beta_{i}+\gamma_{i}\right),\left(\sum^{k}_{i=1}\delta_{i}+\epsilon_{i}\right)\right)}\max_{i=1,\dots , k}{\left(\max{(x_{n-i},y_{n-i})}\right)}.$$
Renaming $z_{n}=\max{(x_{n},y_{n})}$, $b=\max{(\alpha,p)}$, and $C=\max{\left(\left(\sum^{k}_{i=1}\beta_{i}+\gamma_{i}\right),\left(\sum^{k}_{i=1}\delta_{i}+\epsilon_{i}\right)\right)}$, we get the difference inequality
$$z_{n}\leq b+C\max_{i=1,\dots , k}{z_{n-i}},\quad for\;all\;n\in\mathbb{N}.$$
Thus from Theorem 2 in \cite{f1} $\max_{i=1,\dots , k}{z_{n-i}}\leq \max{\left(u_{\left\lfloor \frac{n}{k}\right\rfloor},\dots , u_{n}\right)}$. Where $\{u_{n}\}^{\infty}_{n=1}$ is a solution of the difference equation
$$u_{n}=b+Cu_{n-1}.$$
Since $\sum^{k}_{i=1}\beta_{i}+\gamma_{i}<1$ and $\sum^{k}_{i=1}\delta_{i}+\epsilon_{i}<1$ every solution is bounded above also clearly every solution is bounded below by zero.\newline
Let $S_{1}=\limsup_{n\ra\infty}{x_{n}}$, $I_{1}=\liminf_{n\ra\infty}{x_{n}}$, $S_{2}=\limsup_{n\ra\infty}{y_{n}}$, and $I_{2}=\liminf_{n\ra\infty}{y_{n}}$. Then we have the following
$$S_{1}\leq \frac{\alpha+\left(\sum^{k}_{i=1}\beta_{i}\right)S_{1} + \left(\sum^{k}_{i=1}\gamma_{i} \right) S_{2}}{1+\left(\sum^{k}_{j=1}B_{j}\right) I_{1} + \left(\sum^{k}_{j=1}C_{j}\right) I_{2}} , $$
$$S_{2}\leq \frac{p+\left(\sum^{k}_{i=1}\delta_{i}\right)S_{1} + \left(\sum^{k}_{i=1}\epsilon_{i} \right) S_{2}}{1+\left(\sum^{k}_{j=1}D_{j}\right) I_{1} + \left(\sum^{k}_{j=1}E_{j}\right) I_{2}} , $$
$$I_{1}\geq \frac{\alpha+\left(\sum^{k}_{i=1}\beta_{i}\right)I_{1} + \left(\sum^{k}_{i=1}\gamma_{i} \right) I_{2}}{1+\left(\sum^{k}_{j=1}B_{j}\right) S_{1} + \left(\sum^{k}_{j=1}C_{j}\right) S_{2}} , $$
$$I_{2}\geq \frac{p+\left(\sum^{k}_{i=1}\delta_{i}\right)I_{1} + \left(\sum^{k}_{i=1}\epsilon_{i} \right) I_{2}}{1+\left(\sum^{k}_{j=1}D_{j}\right) S_{1} + \left(\sum^{k}_{j=1}E_{j}\right) S_{2}} . $$
Thus we get
$$\left(\sum^{k}_{j=1}B_{j}\right)I_{1}S_{1}-\alpha \leq \left(\left(\sum^{k}_{i=1}\beta_{i}\right) -1\right)S_{1} + \left(\sum^{k}_{i=1}\gamma_{i}\right)S_{2} - \left(\sum^{k}_{j=1}C_{j}\right)I_{2}S_{1} , $$
$$\left(\sum^{k}_{j=1}E_{j}\right)I_{2}S_{2}-p \leq \left(\left(\sum^{k}_{i=1}\epsilon_{i}\right) -1\right)S_{2} + \left(\sum^{k}_{i=1}\delta_{i}\right)S_{1} - \left(\sum^{k}_{j=1}D_{j}\right)I_{1}S_{2} , $$
$$\left(\sum^{k}_{j=1}B_{j}\right)I_{1}S_{1}-\alpha \geq \left(\left(\sum^{k}_{i=1}\beta_{i}\right) -1\right)I_{1} + \left(\sum^{k}_{i=1}\gamma_{i}\right)I_{2} - \left(\sum^{k}_{j=1}C_{j}\right)I_{1}S_{2} , $$
$$\left(\sum^{k}_{j=1}E_{j}\right)I_{2}S_{2}-p \geq \left(\left(\sum^{k}_{i=1}\epsilon_{i}\right) -1\right)I_{2} + \left(\sum^{k}_{i=1}\delta_{i}\right)I_{1} - \left(\sum^{k}_{j=1}D_{j}\right)I_{2}S_{1} . $$
This gives us
$$\left(\sum^{k}_{j=1}C_{j}\right)\left(I_{2}S_{1}-I_{1}S_{2}\right)\leq \left(\left(\sum^{k}_{i=1}\beta_{i}\right)-1\right)\left(S_{1}-I_{1}\right) + \left(\sum^{k}_{i=1}\gamma_{i}\right)\left(S_{2}-I_{2}\right),$$
$$\left(\sum^{k}_{j=1}D_{j}\right)\left(I_{1}S_{2}-I_{2}S_{1}\right)\leq \left(\left(\sum^{k}_{i=1}\epsilon_{i}\right)-1\right)\left(S_{2}-I_{2}\right) + \left(\sum^{k}_{i=1}\delta_{i}\right)\left(S_{1}-I_{1}\right).$$
We add the inequalities and since $\sum^{k}_{j=1}C_{j}=\sum^{k}_{j=1}D_{j}$ we get
$$0 \leq \left(\left(\sum^{k}_{i=1}\beta_{i}+\delta_{i}\right)-1\right)\left(S_{1}-I_{1}\right)+\left(\left(\sum^{k}_{i=1}\gamma_{i}+\epsilon_{i}\right)-1\right)\left(S_{2}-I_{2}\right).$$
Since $\sum^{k}_{i=1}\beta_{i}+\delta_{i}<1$, and $\sum^{k}_{i=1}\gamma_{i}+\epsilon_{i}<1$, $S_{1}=I_{1}$ and $S_{2}=I_{2}$. Thus every solution converges to a finite limit. 

\end{proof}

\section{The unbounded case}
In this section we handle the unbounded case. The unbounded case proceeds for systems in a similar way to the unbounded case for equations. We choose initial conditions so that the system acts linearly. This implies that whenever the associated linear system is unbounded our system is unbounded. 
\begin{thm}
Consider the recursive system on $[0,\infty)^{m}$
$$v_{n}=B_{n}Av_{n-k},\quad n\in\mathbb{N},$$
where $A=(a_{ij})$ is a real $m\times m$ matrix with nonnegative entries $a_{ij}\geq 0$ and with initial conditions in $[0,\infty)^{m}$. 
Further assume that $B_{n}$ is a real $m\times m$ diagonal matrix with entries $b_{n,ii}=\frac{1}{1+\sum^{\ell}_{j=1}q_{ij}\cdot v_{n-j}}$ for all $n\in\mathbb{N}$. Where the vectors $q_{ij}\in [0,\infty)^{m}$ and $q_{ij}=0$ for all $j=k,2k,3k,\dots$.
If either of the following hold:
\begin{enumerate}
\item The spectral radius of $A$ is greater than $1$.
\item The spectral radius of $A$ is equal to $1$ and $A$ has an eigenvalue $\lambda$ with $|\lambda|=1$ whose algebraic multiplicity exceeds its geometric multiplicity.
\end{enumerate}
Then for some choice of initial conditions the solution $\{v_{n}\}^{\infty}_{n=1}$ is such that  $\{||v_{n}||\}^{\infty}_{n=1}$ is an unbounded sequence.
\end{thm}

\begin{proof}
Before we begin to prove the first case notice that if we choose initial conditions so that $v_{n}=0$ for $n<1$ and $n\neq 1-k$, then it is clear by a simple induction argument that $v_{n}=0$ for $n\not\equiv 1\mod k$. Thus for solutions with these initial conditions we have $v_{n}=Av_{n-k}$. 
We intend to take advantage of this linearity so we will assume that $v_{n}=0$ for $n<1$ and $n\neq 1-k$, and our goal in both cases will be to choose $v_{1-k}$ appropriately in order to create an unbounded solution.\newline
If we choose $v_{1-k}\in [0,\infty )^{m}$ so that for all the generalized eigenvectors of $A$, $w_{1},\dots w_{m}$, $\langle v_{1-k}, w_{i}\rangle \neq 0$ for all $i\in\{1,\dots ,m\}$, this is certainly possible since $[0,\infty)^{m}$ is an m-dimensional subspace of $\mathbb{R}^{m}$.  
Now in case (1) we notice that $||v_{kL+1}||= ||A^{L+1}v_{1-k}||= \sqrt{\langle A^{L+1}v_{1-k},A^{L+1}v_{1-k} \rangle}$, thus $\{||v_{kL+1}||\}^{\infty}_{L=1}$ is unbounded, so $\{||v_{n}||\}^{\infty}_{n=1}$ is unbounded.
Now in case (2) we notice that $||v_{kL+1}||= ||A^{L+1}v_{1-k}||= \sqrt{\langle A^{L+1}v_{1-k},A^{L+1}v_{1-k} \rangle}$, thus $\{||v_{kL+1}||\}^{\infty}_{L=1}$ is unbounded, so $\{||v_{n}||\}^{\infty}_{n=1}$ is unbounded.

\end{proof}

\section{The Hermitian case}
In this section we use the Perron-Frobenius theorem along with our work in the last 2 sections to demonstrate a general periodic trichotomy result.
For more details regarding the Perron-Frobenius theorem see \cite{perronfrob} chapter $8$ sections $2$ and $3$.
Recall that if we have a symmetric matrix with real coefficients then such a matrix must be Hermitian. Any such matrix $A$ is diagonalizable and has decomposition $UDU^{*}$ where $D$ is a diagonal matrix consisting of the eigenvalues of $A$, $U$ is a unitary matrix, and $U^{*}$ represents the conjugate transpose of $U$. Furthermore we know that $D$ has only real entries. The following fact will be useful.
\begin{fact}
Suppose we have a real symmetric $m\times m$ matrix $A$ whose spectral radius is $1$ then $\langle Av,Av\rangle\leq \langle v,v\rangle$ for all $v\in \mathbb{R}^{m}$. 
Moreover $\langle Av,Av\rangle = \langle v,v\rangle$ if and only if $v$ is in the span of the eigenvectors of $A$ with corresponding eigenvalues whose absolute value is $1$. 
\end{fact}
\begin{thm}
Consider the recursive system on $[0,\infty)^{m}$
$$v_{n}=B_{n}Av_{n-k},\quad n\in\mathbb{N},$$
where $A=(a_{ij})$ is a real symmetric $m\times m$ matrix with positive entries $a_{ij}> 0$ and with initial conditions in $[0,\infty)^{m}$. 
Further assume that $B_{n}$ is a real $m\times m$ diagonal matrix with entries $b_{n,ii}=\frac{1}{1+\sum^{\ell}_{j=1}q_{ij}\cdot v_{n-j}}$ for all $n\in\mathbb{N}$. Where the vectors $q_{ij}\in [0,\infty)^{m}$ and $q_{ij}=0$ for all $j=k,2k,3k,\dots$.
Then this system displays the following trichotomy behavior:
\begin{enumerate}[i]
\item If the spectral radius of $A$ is less than $1$ then every solution converges to the zero equilibrium.
\item If the spectral radius of $A$ is equal to $1$ then every solution converges to a solution of not necessarily prime period $k$. Furthermore in this case there exist solutions of prime period $k$.
\item If the spectral radius of $A$ is greater than $1$ then for some choice of initial conditions the solution $\{v_{n}\}^{\infty}_{n=1}$ has the property that $\{||v_{n}||\}^{\infty}_{n=1}$ is an unbounded sequence. Moreover, if we consider the sequences consisting of the entries of $v_{n}$, $\{v_{n,i}\}^{\infty}_{n=1}$, then $\{v_{n,i}\}^{\infty}_{n=1}$ is an unbounded sequence for every $i\in \{1,\dots , m\}$.
\end{enumerate}
\end{thm}
\begin{proof}
First notice that (i) follows immediately from Theorem 1. Now consider case (iii). From Theorem 3 we get immediately that there is some choice of initial conditions so that the solution $\{v_{n}\}^{\infty}_{n=1}$ has the property that $\{||v_{n}||\}^{\infty}_{n=1}$ is an unbounded sequence.
Recall from the proof of Theorem 3 that every unbounded solution we constructed had the property that $v_{n}=0$ for $n<1$ and $n\neq 1-k$. For our purposes we will choose an unbounded solution which has this property, thus $v_{n}=Av_{n-k}$ for our solution. Since $\{||v_{n}||\}^{\infty}_{n=1}$ is an unbounded sequence it follows as a consequence 
$\{v_{n,i_{1}}\}^{\infty}_{n=1}$ is an unbounded sequence for some $i_{1}\in \{1,\dots , m\}$. So there is a subsequence $\{v_{n_{L},i_{1}}\}$ which diverges to $\infty$. Since $A=(a_{ij})$ is a real $m\times m$ matrix with positive entries $a_{ij}> 0$ and $v_{n_{L}+k}=Av_{n_{L}}$, the subsequence
 $\{v_{n_{L}+k,i}\}$ diverges to $\infty$ for all $i\in \{1,\dots , m\}$. So $\{v_{n,i}\}^{\infty}_{n=1}$ is an unbounded sequence for all $i\in \{1,\dots , m\}$. This concludes the proof of case (iii).\newline
To prove case (ii) we use the Perron-Frobenius theorem. The Perron-Frobenius theorem tells us that if $A=(a_{ij})$ is a real $m\times m$ matrix with positive entries $a_{ij}> 0$, then there is a positive real number $r$ called the Perron-Frobenius eigenvalue such that $r$ is an eigenvalue of $A$ and so that any other possibly complex eigenvalue $\lambda$ has $|\lambda|<r$. Moreover $r$ is a simple root of the characteristic polynomial and there is an eigenvector $w_{r}$ associated with $r$ having strictly positive components.
Now combining this with the fact that the spectral radius is $1$ we get that $r=1$ and every other eigenvalue $\lambda$ has $|\lambda|<1$.
Also we know that $r$ is a simple root of the characteristic polynomial so $r$ has algebraic multiplicity equal to $1$. So it must be true that every eigenvalue $\lambda$ with $|\lambda|=1$ has algebraic multiplicity equal to geometric multiplicity. Thus Fact 1 applies in this case.\newline
Since $0\leq b_{n,ii}=\frac{1}{1+\sum^{\ell}_{j=1}q_{ij}\cdot v_{n-j}}\leq 1$ for all $i\in \{1,\dots , m\}$ we have $||v_{n}||\leq ||Av_{n-k}||$.
Fact 1 gives us $||Av||\leq ||v||$ for all $v\in\mathbb{R}^{m}$.
Thus $||v_{n}||\leq ||Av_{n-k}||\leq ||v_{n-k}||$. Since each of the subsequences $\{||v_{nk+a}||\}^{\infty}_{n=1}$ are monotone decreasing and bounded below by zero, they all converge. 
So $\lim_{n\ra\infty}||v_{n}||-||v_{n-k}||=0$. By the squeeze theorem we get $\lim_{n\ra\infty}||v_{n}||-||Av_{n-k}||=0$.\newline
So the subsequences $\{||v_{nk+a}||\}^{\infty}_{n=1}$ and $\{||Av_{nk+a}||\}^{\infty}_{n=1}$ are convergent and since $\lim_{n\ra\infty}||v_{n}||-||Av_{n-k}||=0$  we get
$$\lim_{n\ra\infty}||v_{nk+a}||=\mathfrak{L}_a=\lim_{n\ra\infty}||Av_{nk+a}||.$$
Now consider the sequence $\{v_{nk+a}\}^{\infty}_{n=1}$ and let $\{v_{n_{j}k+a}\}^{\infty}_{j=1}$ be a convergent subsequence with $\lim_{j\ra\infty}v_{n_{j}k+a}=w_a$. By what we have just shown it must be true that
$||w_a||=||Aw_a||$, but then by Fact 1 we have that $w_a$ is in the span of the eigenvectors of $A$ with corresponding eigenvalues whose absolute value is $1$. Recall from the Perron-Frobenius theorem that there is only one such eigenvector and it is $w_{1}$, the eigenvector associated to the eigenvalue $1$.
So $w_a=cw_{1}$, where $c$ is an arbitrary constant, and $||w_a||=\mathfrak{L}_a$, also $w_a\in [0,\infty)^{m}$ as a consequence of our choice of initial conditions. Thus $w_a=w_{1} \left(\frac{\mathfrak{L}_a}{||w_{1}||}\right)$. What this means is that the sequence $\{v_{nk+a}\}^{\infty}_{n=1}$
must converge to $w_a=w_{1} \left(\frac{\mathfrak{L}_a}{||w_{1}||}\right)$. Suppose it does not, then for some $\epsilon >0$ there is a subsequence $\{v_{n_{d}k+a}\}^{\infty}_{d=1}$ so that 
$$\left|\left|v_{n_{d}k+a} - w_{1} \left(\frac{\mathfrak{L}_a}{||w_{1}||}\right)\right|\right|>\epsilon$$ for all $d\in\mathbb{N}$. However we know that $\{v_{n_{d}k+a}\}^{\infty}_{d=1}$ is bounded and so it has a convergent subsequence. This means that 
$\{v_{nk+a}\}^{\infty}_{n=1}$ has a convergent subsequence which does not converge to $w_{1} \left(\frac{\mathfrak{L}_a}{||w_{1}||}\right)$. We have already shown that every convergent subsequence of $\{v_{nk+a}\}^{\infty}_{n=1}$ converges to
$w_{1} \left(\frac{\mathfrak{L}_a}{||w_{1}||}\right)$. Thus we have a contradiction. This proves that the sequence $\{v_{nk+a}\}^{\infty}_{n=1}$
must converge to $w_{1} \left(\frac{\mathfrak{L}_a}{||w_{1}||}\right)$.\newline

Thus every solution must converge to a periodic solution of not necessarily prime period $k$. To construct a solution which is periodic with prime period $k$ we use our eigenvector $w_{1}$ associated with the eigenvalue $1$ having strictly positive components.
We choose initial conditions so that for $n>1$ if $n\not\equiv 0 \mod{k}$ then $v_{n}=0$ and if $n\equiv 0 \mod{k}$ then $v_{n}=w_{1}$. This is a periodic solution of prime period $k$. This concludes our proof.
\end{proof}

\begin{rmk}
Consider the recursive system on $[0,\infty)^{m}$
$$v_{n}=B_{n}Av_{n-k},\quad n\in\mathbb{N},$$
where $A=(a_{ij})$ is a real symmetric $m\times m$ matrix with non-negative entries $a_{ij}\geq 0$ and with spectral radius $1$. Assume initial conditions are in $[0,\infty)^{m}$. 
Further assume that $B_{n}$ is a real $m\times m$ diagonal matrix with entries $b_{n,ii}=\frac{1}{1+\sum^{\ell}_{j=1}q_{ij}\cdot v_{n-j}}$ for all $n\in\mathbb{N}$. Where the vectors $q_{ij}\in [0,\infty)^{m}$ and $q_{ij}=0$ for all $j=k,2k,3k,\dots$.
Further suppose that $A$ has a single eigenvector $w_{1}$ with eigenvalue $1$ and every other eigenvector $w_{i}$ has eigenvalue $\lambda_{i}$ with $|\lambda_{i}|<1$. Then every solution converges to a solution of not necessarily prime period $k$. Furthermore in this case there exist solutions of prime period $k$.
\end{rmk}
\begin{proof}
Identical to the last part of the proof above. 
\end{proof}

\section{A Periodic Tetrachotomy Result}
Now we combine all of our work to give some preliminary examples of periodic tetrachotomy behavior for systems of two rational difference equations.
 
\begin{thm}
Consider the system of two rational
difference equations
$$x_{n}=\frac{\beta_{k}x_{n-k} +\gamma_{k}y_{n-k}}
{1+\sum_{j=1}^{\ell}B_{j}x_{n-j} + \sum_{j=1}^{\ell}C_{j}y_{n-j}},\quad n\in\mathbb{N},$$
$$y_{n}=\frac{\delta_{k}x_{n-k} +\epsilon_{k}y_{n-k}}
{1+\sum_{j=1}^{\ell}D_{j}x_{n-j} + \sum_{j=1}^{\ell}E_{j}y_{n-j}},\quad n\in\mathbb{N},$$
with nonnegative parameters and nonnegative initial conditions. Assume $B_{j}=C_{j}=D_{j}=E_{j}=0$ for $j=k,\; 2k,\; 3k,\; \dots$. Define a matrix
$$A=\left(\begin{array}{cc}
\beta_{k} & \gamma_{k} \\
\delta_{k} & \epsilon_{k} \\
\end{array}\right).$$ 
This system exhibits the following tetrachotomy behavior.
\begin{enumerate}[I]
\item Suppose the spectral radius of $A$ is less than $1$, then every solution converges to the zero equilibrium.
\item Suppose the spectral radius of $A$ is equal to $1$, every eigenvalue $\lambda$ with $|\lambda|=1$ has algebraic multiplicity equal to its geometric multiplicity, and $-1$ is not an eigenvalue of $A$, then every solution converges to a periodic solution of not necessarily prime period $k$. Furthermore in this case there exist periodic solutions with prime period $k$.
\item Suppose the spectral radius of $A$ is equal to $1$, every eigenvalue $\lambda$ with $|\lambda|=1$ has algebraic multiplicity equal to its geometric multiplicity, and $-1$ is an eigenvalue of $A$, then every solution converges to a periodic solution of not necessarily prime period $2k$. Furthermore in this case there exist periodic solutions with prime period $2k$.
\item Suppose the spectral radius of $A$ is greater than $1$ or the spectral radius of $A$ is equal to $1$ and $A$ has an eigenvalue $\lambda$ with $|\lambda|=1$ whose algebraic multiplicity exceeds its geometric multiplicity. Then there exist solutions where $x_{n}+y_{n}$ is unbounded.
\end{enumerate}
\end{thm}

\begin{proof}
To begin we rewrite our system using matrix notation, as was done in Section 4. We let
$$v_{n}=\left(\begin{array}{cc}
x_{n}\\               
y_{n}\\ 
\end{array}
\right),\quad A=\left(\begin{array}{cc}
\beta_{k} & \gamma_{k} \\
\delta_{k} & \epsilon_{k} \\
\end{array}\right),$$  and $$B_{n}=\left(\begin{array}{cc}
\frac{1}{1+\sum_{j=1}^{\ell}a_{j}\cdot v_{n-j}} & 0 \\
0 & \frac{1}{1+\sum_{j=1}^{\ell}q_{j}\cdot v_{n-j}} \\
\end{array}\right),$$
where 
$$a_{j}=\left(\begin{array}{cc}
B_{j}\\               
C_{j}\\ 
\end{array}
\right)\quad and\quad q_{j}=\left(\begin{array}{cc}
D_{j}\\               
E_{j}\\ 
\end{array}
\right).$$
Our system then becomes
$$v_{n}= B_{n}Av_{n-k},\quad n\in\mathbb{N}.$$
Now case (I) follows directly from Theorem 1. Also case (IV) follows directly from Theorem 3. 
Recall that the solutions for the eigenvalues of a $2\times 2$ matrix $A$ can be written as
$$\lambda = \frac{1}{2}\left( tr(A)\pm\sqrt{tr^{2}(A)-4det(A)} \right).$$
This computation is fairly straightforward; it appears as an exercise on page 39 in \cite{matrixanal}.
With our definition of $A$ this becomes
$$\lambda = \frac{1}{2}\left( \beta_{k}+\epsilon_{k}\pm\sqrt{(\beta_{k}-\epsilon_{k})^{2} + 4\gamma_{k}\delta_{k}} \right).$$
Now suppose $\delta_{k},\gamma_{k}>0$ and consider the change of variables $\hat{x}_{n}=\left(\sqrt{\frac{\gamma_{k}}{\delta_{k}}}\right)x_{n}$. 
Under this change of variables we get $\hat{\delta}_{k}=\sqrt{\delta_{k}\gamma_{k}}=\hat{\gamma}_{k}$. Notice that our new matrix $\hat{A}$ is symmetric and has the same eigenvalues as $A$.
Thus, in the case where $A$ is a positive matrix, Theorem 4 applies and gives the result. Also, in the case where $\delta_{k},\gamma_{k}>0$, Remark 1 applies and resolves case (II).
Now we will prove case (III). Suppose $\beta_{k}+\epsilon_{k}>0$ and $-1$ is an eigenvalue. Then we must have 
$$\frac{1}{2}\left( \beta_{k}+\epsilon_{k}+\sqrt{(\beta_{k}-\epsilon_{k})^{2} + 4\gamma_{k}\delta_{k}} \right)>1.$$
However since we have assumed that the spectral radius is $1$ in this case that is impossible. Thus $\beta_{k}+\epsilon_{k}\leq 0$ and we know from assumption that $\beta_{k}+\epsilon_{k}\geq 0$. Thus
$\beta_{k}+\epsilon_{k}= 0$ and in case (III) both $-1$ and $1$ are eigenvalues. So in case (III) we have 
$$A=\left(\begin{array}{cc}
0 & \gamma_{k}\\              
\frac{1}{\gamma_{k}} & 0\\ 
\end{array}
\right). $$
So in case (III) we have the following system of rational difference equations
$$x_{n}=\frac{\gamma_{k}y_{n-k}}
{1+\sum_{j=1}^{\ell}B_{j}x_{n-j} + \sum_{j=1}^{\ell}C_{j}y_{n-j}},\quad n\in\mathbb{N},$$
$$y_{n}=\frac{x_{n-k}}
{\gamma_{k}(1+\sum_{j=1}^{\ell}D_{j}x_{n-j} + \sum_{j=1}^{\ell}E_{j}y_{n-j})},\quad n\in\mathbb{N}.$$
Thus, we have the following recursive inequalities
$$x_{n}\leq x_{n-2k},$$
$$y_{n}\leq y_{n-2k}.$$
So the subsequences $\{y_{n2k+a}\}^{\infty}_{n=1}$ and $\{x_{n2k+a}\}^{\infty}_{n=1}$ are all monotone decreasing and bounded below by zero, so they all converge. Thus we have shown that in case (III) every solution converges to a periodic solution of not necessarily prime period $2k$.
Since in case (III) we have 
$$A=\left(\begin{array}{cc}
0 & \gamma_{k}\\              
\frac{1}{\gamma_{k}} & 0\\ 
\end{array}
\right), $$
 choose initial conditions so that for $n>1$ if $n\not\equiv 0 \mod{k}$ then $v_{n}=0$ and if $n\equiv 0 \mod{2k}$ then $$v_{n}=\left(\begin{array}{cc}
a\\              
b\\ 
\end{array}
\right),$$
where $a,b\in [0,\infty)$ and $a\neq \gamma_{k} b$ and if $n\equiv k \mod{2k}$ then $$v_{n}=\left(\begin{array}{cc}
\gamma_{k}b\\              
\frac{a}{\gamma_{k}}\\ 
\end{array}
\right).$$ Then the solution given by these initial conditions is a periodic solution of prime period $2k$. 
This concludes the proof of case (III).\newline
Thus all we must show to finish case (II) is that when the spectral radius is $1$ and either $\delta_{k}=0$ or $\gamma_{k}=0$ or both, then every solution converges to a periodic solution of prime period $k$. 

Assume that we have $\delta_{k}=\gamma_{k}=0$ in case (II). 
Then we have for $0<\lambda<1$,
$A=\left(\begin{array}{cc}
1 & 0 \\
0 & \lambda  \\
\end{array}\right)$
or
$A=\left(\begin{array}{cc}
\lambda & 0 \\
0 & 1 \\
\end{array}\right).$
Let us focus on the recursive equations for $x_{n}$ and $y_{n}$, we get that
$$x_{n}=\frac{x_{n-k}}
{1+\sum_{j=1}^{\ell}B_{j}x_{n-j} + \sum_{j=1}^{\ell}C_{j}y_{n-j}},\quad n\in\mathbb{N},$$
$$y_{n}=\frac{\lambda y_{n-k}}
{1+\sum_{j=1}^{\ell}D_{j}x_{n-j} + \sum_{j=1}^{\ell}E_{j}y_{n-j}},\quad n\in\mathbb{N}.$$
So we obtain the following recursive inequalities
$$x_{n}\leq x_{n-k},\quad n\in\mathbb{N},$$
$$y_{n}\leq \lambda y_{n-k},\quad n\in\mathbb{N}.$$
So the subsequences $\{x_{nk+a}\}^{\infty}_{n=1}$ and $\{y_{nk+a}\}^{\infty}_{n=1}$ are all monotone decreasing and bounded below by zero, so they all converge and clearly $y_{n}\ra 0$.\newline
Or we have $$x_{n}=\frac{\lambda x_{n-k}}
{1+\sum_{j=1}^{\ell}B_{j}x_{n-j} + \sum_{j=1}^{\ell}C_{j}y_{n-j}},\quad n\in\mathbb{N},$$
$$y_{n}=\frac{ y_{n-k}}
{1+\sum_{j=1}^{\ell}D_{j}x_{n-j} + \sum_{j=1}^{\ell}E_{j}y_{n-j}},\quad n\in\mathbb{N}.$$
So we obtain the following recursive inequalities
$$x_{n}\leq \lambda x_{n-k},\quad n\in\mathbb{N},$$
$$y_{n}\leq  y_{n-k},\quad n\in\mathbb{N}.$$
So the subsequences $\{x_{nk+a}\}^{\infty}_{n=1}$ and $\{y_{nk+a}\}^{\infty}_{n=1}$ are all monotone decreasing and bounded below by zero, so they all converge and clearly $x_{n}\ra 0$.
To construct a periodic solution take the initial conditions so that for $n>1$ if $n\not\equiv 0 \mod{k}$ then $v_{n}=0$ and if $n\equiv 0 \mod{k}$ then $$v_{n}=\left(\begin{array}{cc}
1\\              
0\\ 
\end{array}
\right)\quad or \quad v_{n}=\left(\begin{array}{cc}
0\\              
1\\ 
\end{array}
\right),$$ depending on the case. This is a periodic solution of prime period $k$.\newline
Thus all we must show to finish case (II) is that when the spectral radius is $1$ and either $\delta_{k}=0$ or $\gamma_{k}=0$ but not both, then every solution converges to a periodic solution of prime period $k$. We may assume without loss of generality that $\delta_{k}=0$. If not then make the change of variables $x_{n}=y_{n}$ and vice versa.
Keeping in mind this change of variables we may assume without loss of generality that the only case left to be shown is case (II) when $\delta_{k}=0$ and $\gamma_{k}>0$.
We can now do a further change of variables $\hat{y}_{n}=\frac{y_{n}}{\gamma_{k}}$. Keeping in mind this change of variables we may assume without loss of generality that the only case left to be shown is case (II) when $\delta_{k}=0$ and $\gamma_{k}=1$.
Notice from the eigenvalue calculation earlier that in this case our eigenvalues are $\lambda_{1}=\beta_{k}$ and $\lambda_{2}=\epsilon_{k}$. The spectral radius is $1$ so either $\beta_{k}=1$ or $\epsilon_{k}=1$.
Notice that both $\beta_{k}$ and $\epsilon_{k}$ cannot equal $1$, otherwise we fall into case (IV). This leaves us with $2$ cases. The case where $\beta_{k}=1$ and the case where $\epsilon_{k}=1$.
Let us first consider the case where $\beta_{k}=1$. Focusing on the recursive equations for $x_{n}$ and $y_{n}$ we get that
$$x_{n}=\frac{x_{n-k}+y_{n-k}}
{1+\sum_{j=1}^{\ell}B_{j}x_{n-j} + \sum_{j=1}^{\ell}C_{j}y_{n-j}},\quad n\in\mathbb{N},$$
$$y_{n}=\frac{\epsilon_{k} y_{n-k}}
{1+\sum_{j=1}^{\ell}D_{j}x_{n-j} + \sum_{j=1}^{\ell}E_{j}y_{n-j}},\quad n\in\mathbb{N},$$
where $0\leq \epsilon_{k} < 1$. Now consider the function $h(x,y)=|x|+\left(\frac{1}{1-\epsilon_{k}}\right)|y|$. Then we have 
$$h(x_{n},y_{n})\leq x_{n-k}+y_{n-k}+\frac{\epsilon_{k}y_{n-k}}{1-\epsilon_{k}}= x_{n-k}+ \frac{y_{n-k}}{1-\epsilon_{k}}=h(x_{n-k},y_{n-k}).$$
Notice that since $0\leq \epsilon_{k} < 1$, and $y_{n}\leq \epsilon y_{n-k}$ we have that $y_{n}\ra 0$. Also since $h(x_{n},y_{n})\leq h(x_{n-k},y_{n-k})$ we get that both $x_{n}$ and $y_{n}$ are bounded. Moreover the sequences $\left\{h(x_{nk+a},y_{nk+a})\right\}^{\infty}_{n=1}$ are monotone decreasing and bounded below by zero hence convergent.
So we have $\lim_{n\ra\infty} h(x_{nk+a},y_{nk+a})=\mathfrak{L}_{a}$. 
Now consider the sequence $\{v_{nk+a}\}^{\infty}_{n=1}$ and let $\{v_{n_{j}k+a}\}^{\infty}_{j=1}$ be a convergent subsequence with $\lim_{j\ra\infty}v_{n_{j}k+a}=w_a$. By what we have just shown it must be true that $w_{a}=\left(\begin{array}{cc}
u_{a}\\              
0\\ 
\end{array}
\right)$ for some $u_{a}\geq 0$ and $h(w_{a})=\mathfrak{L}_{a}$. 
This forces $w_{a}=\left(\begin{array}{cc}
\mathfrak{L}_{a}\\              
0\\ 
\end{array}
\right).$ 
What this means is that the sequence $\{v_{nk+a}\}^{\infty}_{n=1}$
must converge to $w_{a}=\left(\begin{array}{cc}
\mathfrak{L}_{a}\\              
0\\ 
\end{array}
\right).$ Suppose it does not, then for some $\epsilon >0$ there is a subsequence $\{v_{n_{d}k+a}\}^{\infty}_{d=1}$ so that 
$$\left|\left|v_{n_{d}k+a} - w_{a}\right|\right|>\epsilon$$ for all $d\in\mathbb{N}$. However we know that $\{v_{n_{d}k+a}\}^{\infty}_{d=1}$ is bounded and so it has a convergent subsequence. This means that 
$\{v_{nk+a}\}^{\infty}_{n=1}$ has a convergent subsequence which does not converge to $\left(\begin{array}{cc}
\mathfrak{L}_{a}\\              
0\\ 
\end{array}
\right).$ We have already shown that every convergent subsequence of $\{v_{nk+a}\}^{\infty}_{n=1}$ converges to
$\left(\begin{array}{cc}
\mathfrak{L}_{a}\\              
0\\ 
\end{array}
\right).$ Thus we have a contradiction. This proves that the sequence $\{v_{nk+a}\}^{\infty}_{n=1}$
must converge to $\left(\begin{array}{cc}
\mathfrak{L}_{a}\\              
0\\ 
\end{array}
\right).$\newline

Thus every solution must converge to a periodic solution of not necessarily prime period $k$. To construct a solution which is periodic with prime period $k$ we choose initial conditions so that for $n>1$ if $n\not\equiv 0 \mod{k}$ then $v_{n}=0$ and if $n\equiv 0 \mod{k}$ then $v_{n}=\left(\begin{array}{cc}
1\\              
0\\ 
\end{array}
\right).$ This is a periodic solution of prime period $k$. This concludes our proof of the case where $\delta_{k}=0$, $\gamma_{k}>0$, and $\beta_{k}=1$.\newline
All that remains is the case where $\delta_{k}=0$, $\gamma_{k}=1$, $\epsilon_{k}=1$, and $0\leq\beta_{k}<1$.
Focusing on the recursive equations for $x_{n}$ and $y_{n}$ we get that
$$x_{n}=\frac{\beta_{k}x_{n-k}+y_{n-k}}
{1+\sum_{j=1}^{\ell}B_{j}x_{n-j} + \sum_{j=1}^{\ell}C_{j}y_{n-j}},\quad n\in\mathbb{N},$$
$$y_{n}=\frac{ y_{n-k}}
{1+\sum_{j=1}^{\ell}D_{j}x_{n-j} + \sum_{j=1}^{\ell}E_{j}y_{n-j}},\quad n\in\mathbb{N},$$
where $0\leq \beta_{k} < 1$. Notice first that since $y_{n}\leq y_{n-k}$, the subsequences $\{y_{nk+a}\}^{\infty}_{n=1}$ with $a\in \{0,\dots ,k-1\}$ are all monotone decreasing and bounded below by $0$, hence they all converge.
Let $\lim_{n\ra\infty}y_{nk+a}=L_{a}$.
Now let $S_{a}$ be the limit superior of the subsequence $\{x_{nk+a}\}^{\infty}_{n=1}$ with $a\in \{0,\dots ,k-1\}$ and let $I_{a}$ be the limit inferior of the subsequence $\{x_{nk+a}\}^{\infty}_{n=1}$ with $a\in \{0,\dots ,k-1\}$. 
This gives us the following
$$S_{a}\leq \frac{\beta_{k}S_{a}+L_{a}}{1+\sum_{j=1}^{\ell}B_{j}I_{(a-j)\mod{k}} + \sum_{j=1}^{\ell}C_{j}L_{(a-j)\mod{k}}},$$
$$I_{a}\geq \frac{\beta_{k}I_{a}+L_{a}}{1+\sum_{j=1}^{\ell}B_{j}S_{(a-j)\mod{k}} + \sum_{j=1}^{\ell}C_{j}L_{(a-j)\mod{k}}}.$$
Thus we have $$-L_{a}\leq S_{a}\left(\beta_{k}-1- \sum_{j=1}^{\ell}C_{j}L_{(a-j)\mod{k}} \right) - \sum_{j=1}^{\ell}B_{j}S_{a}I_{(a-j)\mod{k}},$$ 
$$-L_{a}\geq I_{a}\left(\beta_{k}-1- \sum_{j=1}^{\ell}C_{j}L_{(a-j)\mod{k}} \right) - \sum_{j=1}^{\ell}B_{j}I_{a}S_{(a-j)\mod{k}}.$$
This gives us 
$$0\leq \left(S_{a}-I_{a}\right)\left(\beta_{k}-1- \sum_{j=1}^{\ell}C_{j}L_{(a-j)\mod{k}} \right) + \sum_{j=1}^{\ell}B_{j}\left(I_{a}S_{(a-j)\mod{k}} -S_{a}I_{(a-j)\mod{k}}\right),$$ 
for all $a\in \{0,\dots ,k-1\}$. 
Now notice that $S_{a}\leq \beta_{k}S_{a}+L_{a}$ thus $S_{a}\leq \frac{L_{a}}{1-\beta_{k}}$ for all $a\in \{0,\dots ,k-1\}$. Thus $$I_{a}\geq \frac{L_{a}}{1+\sum_{j=1}^{\ell}\frac{B_{j}}{1-\beta_{k}}L_{(a-j)\mod{k}} + \sum_{j=1}^{\ell}C_{j}L_{(a-j)\mod{k}}}$$ for all $a\in \{0,\dots ,k-1\}$.
Thus if $I_{a}=0$ then $L_{a}=0$ so $S_{a}=0$. So $I_{a}=0$ if and only if $S_{a}=0$. We claim that $S_{a}=I_{a}$ for all $a\in \{0,\dots ,k-1\}$. Assume for the sake of contradiction that this is not the case, then for at least one $a\in \{0,\dots ,k-1\}$, we have $S_{a}>I_{a}>0$. Let $G=\{a\in \{0,\dots ,k-1\}|S_{a}>I_{a}>0\}\neq \emptyset$. 
Consider the element $b\in G$ so that $\frac{S_{b}}{I_{b}}\geq \frac{S_{a}}{I_{a}}$ for all $a\in G$. Such an element must exist since $G$ is finite.
We claim $\left(I_{b}S_{(b-j)\mod{k}} -S_{b}I_{(b-j)\mod{k}}\right)\leq 0$ for all $j\in\mathbb{N}$. Indeed, if $(b-j)\mod{k}\not\in G$ then $S_{(b-j)\mod{k}}=I_{(b-j)\mod{k}}$, so
$$\left(I_{b}S_{(b-j)\mod{k}} -S_{b}I_{(b-j)\mod{k}}\right)= S_{(b-j)\mod{k}}\left(I_{b}-S_{b}\right)\leq 0.$$
Moreover, if $(b-j)\mod{k}\in G$ then $$\frac{S_{b}}{I_{b}}\geq \frac{S_{(b-j)\mod{k}}}{I_{(b-j)\mod{k}}}.$$
Thus $$S_{b}I_{(b-j)\mod{k}}\geq I_{b}S_{(b-j)\mod{k}}.$$
So $$\left(I_{b}S_{(b-j)\mod{k}} -S_{b}I_{(b-j)\mod{k}}\right)\leq 0.$$
Now using the earlier inequality with $b$ we get
$$0\leq \left(S_{b}-I_{b}\right)\left(\beta_{k}-1- \sum_{j=1}^{\ell}C_{j}L_{(a-j)\mod{k}} \right) + \sum_{j=1}^{\ell}B_{j}\left(I_{b}S_{(b-j)\mod{k}} -S_{b}I_{(b-j)\mod{k}}\right)$$
$$\leq \left(S_{b}-I_{b}\right)\left(\beta_{k}-1- \sum_{j=1}^{\ell}C_{j}L_{(a-j)\mod{k}} \right).$$
This forces $S_{b}=I_{b}$, but we chose $b\in G$. This is a contradiction. This establishes the claim $S_{a}=I_{a}$ for all $a\in \{0,\dots ,k-1\}$. Thus all of the subsequences $\{x_{nk+a}\}^{\infty}_{n=1}$ with $a\in \{0,\dots ,k-1\}$ converge. Thus, every solution converges to a periodic solution of not necessarily prime period $k$.
To construct a solution which is periodic with prime period $k$ we choose initial conditions so that for $n>1$ if $n\not\equiv 0 \mod{k}$ then $v_{n}=0$ and if $n\equiv 0 \mod{k}$ then $v_{n}=\left(\begin{array}{cc}
\frac{1}{1-\beta_{k}}\\              
1\\ 
\end{array}
\right).$ This is a periodic solution of prime period $k$. This concludes our proof of case (II) and the theorem is proved.

\end{proof}

\section{Conclusion}
We have created some analogues for trichotomy behavior for systems of rational difference equations, but we have barely scratched the surface. There are literally thousands of special cases of systems of rational difference equations of order greater than one to explore.
This paper leaves several questions for further study. Are there any other examples of periodic tetrachotomy behavior for systems of two rational difference equations? Is it possible to make analogues to other trichotomy results in the literature?
The main idea to take away from this article is that in some cases it is useful to reframe a problem about systems of rational difference equations as a problem about recursive systems on vector spaces. Doing this allows one to utilize the powerful tools available in linear algebra.
Further work should focus on proving a similar result in systems of three rational difference equations. Note that the results in \cite{f2} and \cite{f3} susbsume and unify a number of prior results. For example the case presented in section 2 is a minor generalization of a case originally presented in \cite{kst}. We list, for the readers convienience, the references \cite{aggl},\cite{book}, and [6-22] as these references provide good background on trichotomy character for rational difference equations. 
\vfill
\pagebreak
\par
\end{document}